\documentclass[11pt, reqno]{amsart}
\usepackage{amssymb,latexsym,amsmath,amsfonts}
\usepackage[mathscr]{eucal}
\numberwithin{equation}{section}
\theoremstyle{definition}
\newtheorem{definition}{Definition}[section]
\newtheorem{remark}[definition]{Remark}
\theoremstyle{plain}
\newtheorem{theorem}[definition]{Theorem}
\newtheorem{lemma}[definition]{Lemma}

\newtheorem{example}[definition]{Example}
\newtheorem{conjecture}[definition]{Conjecture}
\newtheorem{result}[definition]{Result}

\newcommand{\eps}{\varepsilon}

\newcommand{\zt}{\zeta}
\newcommand{\zbar}{\overline{z}}
\newcommand{\tht}{\theta}
\newcommand{\zahl}{\mathbb{Z}}  
\newcommand{\nat}{\mathbb{N}}

\newcommand{\rem}{\mathcal{R}}

\newcommand\partl[2]{\dfrac{\partial{#1}}{\partial{#2}}}

\newcommand{\bdy}{\partial}
\newcommand{\dom}{\mathscr{D}}
\newcommand{\OM}{\Omega}
\newcommand{\dee}{\mathbb{D}}

\newcommand{\smoo}{\mathcal{C}}
\newcommand{\hol}{\mathcal{O}}

\newcommand{\er}{{\sf Re}}
\newcommand{\mi}{{\sf Im}} 
\newcommand{\lead}{\mathscr{F}_m}
\newcommand{\leadC}{\mathscr{F}_C}
\newcommand{\derivF}{\mathcal{Q}_m}
\newcommand{\dellead}{\mathfrak{p}_m}
\newcommand{\MasInd}{{\rm Ind}_{\mathcal{M}}}
\newcommand\pathInd[1]{{\rm Ind}_{\mathcal{M},{#1}}}
\newcommand\discInd[1]{{\rm Ind}_{\mathcal{M},{#1}}}
\newcommand{\impl}{\Longrightarrow}
\newcommand{\lrarw}{\longrightarrow}
\newcommand{\bcdot}{\boldsymbol{\cdot}}
\newcommand{\grass}{G_{tot.\mathbb{R}}}
\newcommand\cis[1]{e^{i{#1}}}
\newcommand{\cnj}{\mathfrak{H}}
\newcommand{\bish}{\mathscr{A}}
\newcommand{\invrt}{\mathbb{A}}

\newcommand\hnrm[1]{\left\|{#1}\right\|_{\smoo^{\alpha}}}
\newcommand\hlpart[1]{\left[{#1}\right]_{\alpha}}
\newcommand\uninrm[1]{\left\|{#1}\right\|_{\infty}}
\newcommand\smoonrm[2]{\left\|{#1}\right\|_{\smoo^{{#2}}(\overline{\dee})}}

\newcommand{\CC}{\mathbb{C}^2}
\newcommand{\Cn}{\mathbb{C}^n}
\newcommand{\cplx}{\mathbb{C}} 

\newcommand{\rl}{\mathbb{R}}

\newcommand{\srf}{\mathcal{S}}   

\begin{document}

\title[The polynomial hull and Bishop discs]{The local polynomial hull near a degenerate \\
CR~singularity -- Bishop discs revisited}

\author{Gautam Bharali}

\thanks{This work is supported by the DST via the Fast Track grant SR/FTP/MS-12/2007 and
by the UGC under DSA-SAP, Phase~IV}

\address{Department of Mathematics, Indian Institute of Science, Bangalore -- 560012, India}
\email{bharali@math.iisc.ernet.in}


\begin{abstract}
Let $\srf$ be a smooth real surface in $\CC$ and let $p\in\srf$ be a point at which the
tangent plane is a complex line. How does one determine whether or not
$\srf$ is locally polynomially convex at such a $p$ --- i.e. at a
CR~singularity\,? Even when the order of contact of $T_p(\srf)$ with $\srf$ at $p$ equals $2$,
no clean characterisation exists; difficulties are posed by parabolic points. Hence, we
study {\em non-parabolic} CR~singularities. We show that the presence or
absence of Bishop discs around certain non-parabolic CR~singularities is completely determined
by a Maslov-type index. This result subsumes all known facts about Bishop discs around order-two,
non-parabolic CR~singularities. Sufficient conditions for Bishop discs have earlier been
investigated at CR~singularities having high order of contact with $T_p(\srf)$. These results
relied upon a subharmonicity condition, which fails in many simple cases. Hence, we look 
beyond potential theory and refine certain ideas going back to Bishop. 
\end{abstract}
\maketitle

\section{Introduction and statement of results}\label{S:intro}

The simplest motive for this work is a rather naive one: we would like to know when, given 
a real surface $\srf\subset\CC$ and a point $p\in\srf$ at which $T_p(\srf)$ is a complex line,
$\srf$ is locally polynomially convex at $p$. For $\srf$ having only isolated exceptional points,
this knowledge would enable one to determine whether $\srf$ has a Stein neighbourhood basis.
Insights into this naive problem would enable one to make tangible use of the many results about
polynomial approximation on compact $2$-submanifolds $\srf\subset\CC$ {\em with boundary}, 
most of which presuppose the polynomial convexity of $\srf$.
\smallskip

We shall call a point of complex tangency a {\em CR~singularity}. Consider a
CR~singularity $p\in\srf\subset\CC$ where the order of  contact of $T_p(\srf)$ with $\srf$ equals 
$2$ --- i.e. a {\em non-degenerate} CR~singularity. Bishop showed \cite{bishop:dmcEs65} that there exist
holomorphic coordinates $(z,w)$ centered at $p$ such that $\srf$ is locally given (barring one manifestly
locally polynomially convex case) by the equation $w=|z|^2+\gamma(z^2+\zbar^2)+G(z)$, 
where $\gamma\geq 0$, $G(z)=O(|z|^3)$, and three distinct situations arise. In Bishop's 
terminology, the CR~singularity $p=(0,0)$ is called elliptic if 
$0\leq\gamma<1/2$, parabolic if $\gamma=1/2$, and hyperbolic if $\gamma>1/2$. Bishop \cite{bishop:dmcEs65}
showed that when $p\in\srf$ is elliptic, the polynomially convex hull of $\srf$ near $p$ contains 
a one-parameter family of non-constant analytic discs attached to $\srf$ that shrink to $p$.   
On the other hand, Forstneri{\v{c}} and Stout \cite{forstnericStout:ncpcs91} showed that when $p$ 
is hyperbolic, $\srf$ is locally polynomially convex at $p$. In the applications hinted at, {\em we
may not have the option of perturbing the given $\srf$ at all}, whence the genericity of 
non-degenerate CR singularities cannot aid the study of such applications. Given this, one might ask 
what we can say about $(\srf,p)$ if $p$ is a {\em degenerate} CR~singularity.

Even if $p$ is a CR~singularity in $\srf$ where the order of  contact of $T_p(\srf)$ with 
$\srf$ equals $2$, J{\"o}ricke's results in \cite{joricke:lphdnipp97} show that the situation is
far from tidy when $p$ is a parabolic point. 
One would expect some assumptions on the pair $(\srf,p)$ (for $p$ a degenerate 
CR~singularity) for the outlines of a reasonable pattern, consistent with what is already
known, to emerge. This motivates the following:

\begin{definition}\label{D:non-para}
Let $\srf$ be a $\smoo^k$-smooth real
surface in $\cplx^2$, $k\geq 3$, and let $p\in\srf$ be an isolated CR~singularity.
We say that $p$ is {\em non-parabolic} if there exist an integer $m$, $2\leq m<k$, and 
holomorphic coordinates $(z,w)$ centered at $p$ relative to which $\srf$ has a local 
defining equation
\begin{equation}\label{E:normform}
\srf\cap U_p: \quad w \ = \ \lead(z,\zbar) + \rem(z)
\end{equation}
such that the graph $\Gamma(\lead)$ has an isolated CR singularity at $(0,0)\in\CC$. 
Here, $\lead$ is a homogeneous polynomial in $z$ and $\zbar$ of degree $m$ and
$\rem$ is $O(|z|^{m+1})$.
\end{definition}
\smallskip

Note that when $p\in\srf$ is either elliptic or hyperbolic, it is non-parabolic in the sense of
Definition~\ref{D:non-para}. We wish to extend the Bishop/Forstneri{\v{c}}--Stout
dichotomy (for non-parabolic, non-degenerate CR~singularities) to the degenerate setting. When
$(\srf,p)$ is presented in the Bishop normal form near a non-parabolic, non-degenerate $p$,
we have holomorphic coordinates $(z,w)$ in which  --- using the notation of 
\eqref{E:normform} --- $\mathscr{F}_2$ is real-valued. This last fact is of central
importance to Bishop's proofs in \cite[Section~3]{bishop:dmcEs65}. This motivates the following:

\begin{definition}\label{D:thin}
Let $\srf$ be a $\smoo^k$-smooth real surface in $\cplx^2$, $k\geq 3$, 
and let $p\in\srf$ be an isolated CR~singularity.   
Suppose $T_p(\srf)$ has finite order of contact $2\leq m<k$ with $\srf$ at $p$. 
We say that {\em $\srf$ is thin at $p$} if there exist holomorphic coordinates $(z,w)$ centered at 
$p$ such that $\srf$ is locally a graph of the form \eqref{E:normform} above, and with respect to
which $\lead$ ($\lead$ has the same meaning as in Definition~\ref{D:non-para}) is real-valued.
\end{definition}
\smallskip

When, for the pair $(\srf,p)$, $p$ is a non-parabolic, non-degenerate CR~singularity (in which case
$\srf$ is {\em always} thin at $p$) the works \cite{bishop:dmcEs65} and \cite{forstnericStout:ncpcs91},
when read together, imply that the local polynomial convexity of $\srf$ at $p$ is determined
precisely by the sign of a certain Maslov-type index, denoted by $\MasInd(\srf,p)$. 
Specifically:
\begin{itemize}
\item[$(*)$] {\em  When $p\in\srf$ is a non-parabolic, non-degenerate (hence thin) CR~singularity, 
$\srf$ is locally polynomially convex at $p$ if and only if $\MasInd(\srf,p)\leq 0$.}
\end{itemize}
The reader is directed to Section~\ref{S:Maslov} for the precise definition of the index 
$\MasInd(\srf,p)$. The goal of this paper is to attempt to extend $(*)$ to non-parabolic, degenerate
CR~singularities. That brings us to our first theorem, which says, among other things, that 
$\MasInd(\srf,p)>0\impl \srf$ is {\em not} locally polynomially convex at $p$. 

\begin{theorem}\label{T:posInd}
Let $\srf$ be a $\smoo^k$-smooth real surface in $\CC$, $k\geq 3$,  and 
let $p\in\srf$ be a CR~singularity. Assume that $p$
is non-parabolic and that $\srf$ is thin at $p$. If $\MasInd(\srf,p)>0$, then $\srf$ is {\em not}
locally polynomially convex at $p$. 
\smallskip

In fact, there exists a $\smoo^1$-smooth family of analytic discs whose boundaries
are contained in $\srf$. More precisely: there exist a neighbourhood $U_p\ni p$, an open 
interval $(0,R_0)$, and a
function $\mathfrak{g}:(0,R_0)\lrarw A^\alpha(\dee;\cplx^2)$ that is of class 
$\smoo^1$ on $(0,R_0)$ (for an arbitrary but fixed $\alpha\in (0,1)$), where each 
$\mathfrak{g}(r)$ is a non-constant analytic disc satisfying
\begin{enumerate}
\item[$i)$] $\mathfrak{g}(r)(\bdy\dee)\subset(\srf\setminus\{p\})\cap U_p \ \forall r\in(0,R_0)$; and
\item[$ii)$] $\mathfrak{g}(r)(\zt)\lrarw\{p\}$ for each $\zt\in\overline{\dee}$ as $r\lrarw 0^+$.
\end{enumerate}
\end{theorem}

\noindent{Here, and elsewhere in this paper,
$A^\alpha(\dee;\cplx^2)$ denotes the class of all $\cplx^2$-valued maps on $\overline{\dee}$ that
are holomorphic on $\dee$ and whose restrictions to $\bdy\dee$ are of H{\"o}lder class 
$\smoo^\alpha(\bdy\dee)$.}

\begin{remark}
The term ``thin'' must not be confused with the term ``flat'', which appears in the literature
on polynomial convexity. $\srf$ would be {\em flat} at $p$ if $\rem$ (as given by 
\eqref{E:normform}) were {\em also} real-valued. The term ``thin'' arises in some parts of
the literature on potential theory, but is unrelated to Definition~\ref{D:thin}.
\end{remark}    
\smallskip

Before commenting on the relationship between Theorem~\ref{T:posInd}
and other results on the same theme in the literature, let us present a partial converse of
Theorem~\ref{T:posInd}. 

\begin{theorem}\label{T:negInd}
Let $\srf$ be a $\smoo^k$-smooth real surface in $\CC$, $k\geq 3$, and 
let $p\in\srf$ be a CR~singularity. Assume that $p$
is non-parabolic and that $\srf$ is thin at $p$. Suppose $\MasInd(\srf,p)\leq 0$.
\begin{enumerate}
\item[1)] Let $(z,w)$ be holomorphic coordinates centered at $p$ such that
(by hypothesis) $\srf$ is locally defined by
\begin{equation}\label{E:normform2}
\srf\cap U_p: \quad w \ = \ \lead(z,\zbar) + \rem(z) 
 \quad(\text{$\rem$ is $O(|z|^{m+1})$ for $|z|$ small}),
\end{equation} 
where $\lead$ is a real-valued polynomial that is homogeneous of degree $m$, $2\leq m<k$.
If $\rem$ is real-valued, then $\srf$ is locally polynomially convex at $p$.
\item[2)] Suppose $k\geq 4$.
Given any $\alpha\in(0,1)$ (now $\rem$ need not be real-valued), 
it is impossible to find a continuous
one-parameter family $\mathfrak{g}:(0,1)\lrarw A^\alpha(\dee;\CC)$ of immersed, non-constant
analytic discs having all the following properties:
\begin{itemize}
\item $\mathfrak{g}(t)(\bdy\dee)\subset(\srf\setminus\{p\})\cap U_p \ \forall t\in(0,1)$;
\item $\mathfrak{g}(t)(\cis{\bcdot})$ is a simple closed curve in $\srf \
\forall t\in(0,1)$; and
\item $\mathfrak{g}(t)(\zt)\lrarw\{p\}$ for each $\zt\in\overline{\dee}$ as $t\lrarw 0^+$.
\end{itemize}
\end{enumerate}
\end{theorem}
\smallskip

The point of Part~(2) of Theorem~\ref{T:negInd} is to observe that, although we do not know whether  
$\MasInd(\srf,p)\leq 0$ implies that $\srf$ is locally polynomially convex at $p$, the local
polynomially convex hull of $(\srf,p)$ does not contain non-constant analytic discs (with boundaries
in $\srf\setminus\{p\}$) that shrink to $p$. Note also that each part of Theorem~\ref{T:negInd} 
can be viewed as a partial converse to Theorem~\ref{T:posInd}. These lead us to suggest the 
following conjecture:

\begin{conjecture}\label{C:iff}
Let $\srf$ be a $\smoo^k$-smooth real surface in $\CC$, $k\geq 3$, and 
let $p\in\srf$ be a CR~singularity. Assume that $p$
is non-parabolic and that $\srf$ is thin at $p$. Then, $\srf$ is locally polynomially convex at
$p$ if and only if $\MasInd(\srf,p)\leq 0$.
\end{conjecture}
\smallskip

The above conjecture may remind the reader of the findings of J{\"o}ricke \cite{joricke:lphdnipp97}
and  Wiegerinck \cite{wiegerinck:lpchdcrs95} on {\em parabolic} (non-degenerate) CR~singularities.
At least when $\MasInd(\srf,p)\neq 0$ parabolic points have been shown in 
\cite{wiegerinck:lpchdcrs95, joricke:lphdnipp97} to exhibit the conjectured dichotomy. The
question arises as to why the ideas in \cite{wiegerinck:lpchdcrs95, joricke:lphdnipp97}
should not reveal the same dichotomy when applied to non-parabolic, degenerate CR~singularities.
But the fact is, {\em without the condition of thinness, the dichotomy does not hold:} Wiegerinck 
has given examples \cite[Section 4]{wiegerinck:lpchdcrs95} in which $\MasInd(\srf,p)<0$ and 
yet local polynomial convexity at $p$ fails. The CR~singularities 
studied by Wiegerinck are, for the most part, also non-parabolic, degenerate CR~singularities.
But, in place of thinness, they are required to satisfy a different analytical condition. 
The key point of departure of this article from \cite{wiegerinck:lpchdcrs95}
is summarised by these two observations:
\begin{itemize}
\item Although the surfaces $(\srf,p)$ studied in \cite{wiegerinck:lpchdcrs95} are not 
necessarily thin at $p$,\linebreak
{\em Wiegerinck's hypotheses do not hold true in general when $\srf$ is thin at $p$}.
\item Theorems~\ref{T:posInd} and \ref{T:negInd} provide some evidence in support of
Conjecture~\ref{C:iff}. In contrast, there does not seem to be a clear-cut
discriminant for local polynomial convexity if thinness is replaced by the hypotheses 
in \cite{wiegerinck:lpchdcrs95}.
\end{itemize}
It is true that the class of pairs $(\srf,p)$ with $\srf$ being thin at $p$ forms a small 
sub-case of the general situation. However, {\em this paper is devoted to studying a certain
dichotomy.} Wiegerinck's examples suggest that very different considerations must apply when
$(\srf,p)$ is {\em not} thin. These considerations have been examined --- although more from
the viewpoint of detecting polynomial convexity than of polynomial hulls ---  in 
\cite{bharali:sdCRslpc05} and in a recent article \cite{bharali:palpcdCRsII}.
\smallskip
 
As for the assumptions in Wiegerinck's work: we refer the reader to 
\cite[Theorems~3.3, 3.4]{wiegerinck:lpchdcrs95}. His assumptions, applied to our context,
translate to the requirement that $\lead$ must be subharmonic and non-harmonic. One of 
the motivations of this paper is to develop
tools to show the existence of Bishop discs {\em in the absence of such subharmonicity
conditions}. This is a meaningful motivation because of the following:
\smallskip

\noindent{{\bf Fact} (see Example~\ref{Ex:nonSubh}){\bf .} {\em There exist polynomials
$\lead:\cplx\lrarw\rl$, homogeneous of degree $m$, such that 
\begin{itemize}
\item $0$ is an isolated CR~singularity of $\Gamma(\lead)$ satisfying 
$\MasInd(\Gamma(\lead),0)>0$; and
\item $\lead$ is ${\rm not}$ subharmonic.
\end{itemize}}}

\noindent{Example~\ref{Ex:nonSubh} rules out the possibility of simply applying the results of
\cite{wiegerinck:lpchdcrs95} to deduce Theorem~\ref{T:posInd}.}
\smallskip

Before proceeding to the proofs, we would like to point out a couple of new inputs required in 
the proof of Theorem~\ref{T:posInd}, and to sketch the main ingredients of our approach. Our proof
consists of the following parts:
\begin{itemize}
\item {\bf Part~I.} We work in the coordinate system $(z,w)$ centered at $p$ in which 
$(\srf,p)$ is presented locally as shown in \eqref{E:normform2}. We prove a general result:
\[
\MasInd(\Gamma(\lead),0) \ = \ 
-\frac{\text{\#}[\lead(e^{i\bcdot})^{-1}\{0\}\cap[0,2\pi)]}{2}+1
\]
(the notation $\text{\#}[S]$ stands for the cardinality of the set $S$). 
This tells us, since $\MasInd(\srf,p)>0$, that we may assume (after making a holomorphic
change of coordinate if necessary) that $\lead(z)>0 \ \forall z\in\cplx\setminus\{0\}$.

\item {\bf Part~II.} We see that $\lead^{-1}\{1\}$ is a simple closed real-analytic curve. Let
$g$ denote the boundary-value of the normalised Riemann mapping of $\dee$ onto the region enclosed
by $\lead^{-1}\{1\}$. Then, the curves $\varphi_r:\partial\dee\lrarw\CC, \ r>0$, given by
$\zt\longmapsto(rg(\zt),r^m)$ are closed curves in $\Gamma(\lead)$ that bound analytic discs.
We view $\srf$, equivalently the graph $\Gamma(\lead+\rem)$, as a small perturbation of
$\Gamma(\lead)$, and attempt to obtain small corrections, say ${\psi_r}$, 
of $\varphi_r \ \forall r\in(0,R_0)$, for $R_0>0$ sufficiently
small, such that $(\varphi_r+\psi_r)$ are curves in $\Gamma(\lead+\rem)$ that bound analytic discs.
This requirement gives us a family of functional equations,
involving the harmonic-conjugate operator, parametrised
by the interval $(0,R_0)$. The desired $\psi_r, \ r\in(0,R_0)$, are derived from the fixed points 
of these equations.

\item {\bf Part~III.} One way to obtain fixed points is to show that the functionals
involved in the aforementioned equations are contractions. This is the approach of Kenig~\&~Webster 
in \cite{kenigWebster:lhhss2Cv82}. In making the required estimates, Kenig and Webster are aided by
the following remarkable fact:
\begin{itemize}
\item[$(\blacktriangle)$] {\em If, in addition to the hypotheses in
Theorem~\ref{T:posInd}, the polynomial $\lead$ is quadratic, then given any $l\in\nat$, $l\geq 3$,
there exists a holomorphic coordinate system $(z,w)$ such that $(\srf,p)$ has a local representation
of the form \eqref{E:normform2} and such that $\mi(\rem)(z)=O(|z|^{l+1})$.}
\end{itemize}
This fact is good enough to show that the Bishop discs foliate a $\smoo^\infty$-smooth $3$-manifold
with boundary. Unfortunately, {\em the conclusion of $(\blacktriangle)$ is not true in general if
$m>2$}. In the absence of $(\blacktriangle)$, we just make more stringent estimates. These estimates
turn out to be good enough to conclude that $(0,R_0)\ni r\longmapsto(\varphi_r+\psi_r)$ is 
$\smoo^1$-smooth.
\end{itemize}
\smallskip

One final expository remark is in order: one could set up a functional equation of the type 
that we allude to in Part~II above, and naively hope to show that $r\longmapsto\psi_r$ is of
class $\smoo^1$ using the Implicit Function Theorem. The problem is that, owing to the presence of
the CR singularity, the relevant Fr{\'e}chet (partial) derivative of the non-linear functional
involved is {\em non-surjective at all the obvious zeros of this functional!} The reader's 
attention is drawn to the note in Step~2 of Section~\ref{S:posInd}. It is this fact that leads
to the (unavoidable) technicalities of the approach outlined above.
\smallskip

Since an important part of both Theorems~\ref{T:posInd} and~\ref{T:negInd} is based on a good
understanding of $\MasInd(\srf,p)$, we shall begin with a discussion on this index in the next
section. The proofs of Theorems~\ref{T:posInd} and~\ref{T:negInd} will be presented in 
Sections~\ref{S:posInd} and~\ref{S:negInd} respectively. A discussion on the non-subharmonicity
of the local graphing functions of the $(\srf,p)$ that we consider in this paper will be
presented in Section~\ref{S:nonSubh}. 
\medskip

\section{Some facts about the Maslov-type index}\label{S:Maslov} 

Given a smooth real surface $\srf\subset\CC$, the term ``Maslov-type index'' might refer to 
three inter-related numbers that apply to slightly different contexts. They are:
\begin{enumerate}
\item[{a)}] {\em The index $\pathInd{\gamma}(\srf)$ of a closed path:} This applies to a
closed path $\gamma:S^1\lrarw\srf$, where $\srf$ is a totally-real submanifold of a region 
$\OM\subseteq\CC$.
\smallskip

\item[{b)}] {\em The index $\MasInd(\srf,p)$ of a CR~singularity $p$:} This
applies to a pair $(\srf,p)$, where $\srf$ is an orientable real $2$-submanifold of some 
region $\OM\subseteq\CC$ having an isolated CR~singularity at $p\in\srf$.
\smallskip

\item[{c)}] {\em The index $\discInd{\psi}(\srf)$ of an analytic disc 
$\psi$:} This applies to an analytic disc $\psi\in\hol(\dee;\CC)\cap\smoo(\overline\dee;\CC)$
with $\psi(\bdy\dee)\subset\srf$, where $\srf$ is a totally-real submanifold of some region
$\OM\subseteq\CC$.
\end{enumerate} 
\smallskip

\noindent{In this paper, it is the first two indices that will be relevant to our discussions.
Before making the proper definitions, we will need one piece of notation.
We set
\[
\grass(\CC) \ := \ \text{the manifold of oriented totally-real planes in $\CC$},
\]
where the differentiable structure on $\grass(\CC)$ is the one that makes it a 
submanifold of the Grassmanian $G(2,\mathbb{R}^4)$ of oriented $2$-subspaces of
$\mathbb{R}^4$. We are now in a position to make our definitions. In doing so, we follow
the constructions by Forstneri{\v{c}} in 
\cite{forstneric:adbmrsCC87}. Here, we make
one remark: we wish to define the concepts (a) and (b) above with the least amount of
technicality possible, and to draw upon some computations in \cite{forstneric:adbmrsCC87}
that pertain to graphs in $\CC$. Hence, in the definitions
below {\em we will assume that the bundle $\gamma^*\left.T\srf\right|_{\gamma(S^1)}$ is
a trivial bundle} (where $\gamma:S^1\lrarw\srf$ is as in (a)), although the notion of
$\pathInd{\gamma}(\srf)$ is not restricted to the trivial-bundle case.

\begin{definition}\label{D:pathInd}
Let $\srf$ be a totally-real $2$-submanifold of a region $\OM\subseteq\CC$. 
Let $\gamma:S^1\lrarw\srf$ be a smooth, closed path such that the pullback 
$\gamma^*\left.T\srf\right|_{\gamma(S^1)}$ is a 
trivial bundle (equivalently, $\srf$ is orientable along $\gamma$). 
Let $\Theta_\gamma:S^1\lrarw\grass(\CC)$ denote the tangent map, i.e.
$\Theta_\gamma(\zt):=T_{\gamma(\zt)}(\srf)$. There is a well-defined map 
$\mathfrak{G}:\grass(\CC)\lrarw\cplx\setminus\{0\}$ given by
\[
\mathfrak{G}(P):={\rm det}\left[X^P_1 \;\; X^P_2\right], \;\; (X^P_1,X^P_2) \ \text{\em a positively 
oriented orthonormal basis of $P$},
\]
this being well-defined because, given two positively oriented orthonormal bases
$(X^P_1,X^P_2)$ and $(Y^P_1,Y^P_2)$, $Y^P_j=A(X^P_j), \ j=1,2$, for some 
$A\in SL(2,\mathbb{R})$. The composition $\mathfrak{G}\circ\Theta_\gamma$ 
induces a homomorphism in homology $H_1(\mathfrak{G}\circ\Theta_\gamma):H_1(S^1;\zahl)\lrarw 
H_1(\cplx\setminus\{0\};\zahl)$. The degree of this homomorphism is called the 
{\em Maslov-type index of the path $\gamma$}, denoted by $\pathInd{\gamma}(\srf)$. 
\end{definition}  

\begin{definition}\label{D:MasInd}
Let $\srf$ be a real orientable $2$-submanifold of some region $\OM\subseteq\CC$ that has
an isolated CR~singularity at $p\in\srf$. Then there is an $\srf$-open neighbourhood of $p$,
say $W_p$, that is contractible to $p$ and such that $p$ is the only CR~singularity in $W_p$.
Let $W_p$ have the orientation induced by the complex line $T_p(\srf)$. 
Let $\gamma:S^1\lrarw W_p\setminus\{p\}$ be a smooth, simple
closed curve that has positive orientation with respect to the orientation of $W_p$. Then,
we define the {\em Maslov-type index of the CR~singularity $p$}, written as 
$\MasInd(\srf,p)$, by $\MasInd(\srf,p):=\pathInd{\gamma}(W_p\setminus\{p\})$.
\end{definition}
\smallskip

We note that $\MasInd(\srf,p)$ is well-defined because 
$\pathInd{\gamma}(W_p\setminus\{p\})$ depends only on 
the homology class of $\gamma$ in $W_p\setminus\{p\}$. When $\srf$ is the graph $\Gamma(F)$
of some function $F$ that is $\smoo^1$-smooth near $0\in\cplx$, with $\Gamma(F)$ having an
isolated CR singularity at the origin, then $\gamma^*\left.T\srf\right|_{\gamma(S^1)}$ is trivial
for any $\gamma:S^1\lrarw\Gamma(F)\setminus\{(0,0)\}$ as in Definition~\ref{D:MasInd}. Using
an explicit frame for  $\gamma^*\left.T\Gamma(F)\right|_{\gamma(S^1)}$, Forstneri{\v{c}} has
shown that:

\begin{lemma}[\cite{forstneric:adbmrsCC87}, Lemma~8]\label{L:windex}
Let $\OM$ be a domain in $\cplx$ containing $0$ and let $F\in\smoo^1(\OM;\cplx)$. Suppose
that the graph $\Gamma(F)$ has an isolated CR~singularity at $0$. Let 
$\gamma:S^1\lrarw \OM\setminus\{0\}$ be a smooth, positively-oriented, simple closed curve
that encloses $0$ and encloses no other points belonging to $(\partial F/\partial\zbar)^{-1}\{0\}$.
Then
\begin{equation}\label{E:windex}
\MasInd(\Gamma(F),0) \ = \ {\sf Wind}\left(\partl{F}{\zbar}\circ\gamma,0\right),
\end{equation}
where the expression on the right-hand side denotes the winding number around $0$.
\end{lemma}
\smallskip

We are now in a position to prove a key lemma. This was informally stated in Part~I of our outline,
in Section~\ref{S:intro}, of the proof of Theorem~\ref{T:posInd}.

\begin{lemma}\label{L:key}
Let $\lead:\cplx\lrarw\rl$ be a polynomial that is homogeneous of degree $m$ and such that 
$(\partial \lead/\partial\zbar)^{-1}\{0\}=\{0\}$. Then
\begin{equation}\label{E:keyWindex}
\MasInd(\Gamma(\lead),0) \ = \ -\frac{\text{\#}[\lead(e^{i\bcdot})^{-1}\{0\}\cap[0,2\pi)]}{2}+1
\end{equation}
(the notation $\text{\#}[S]$ denotes the cardinality of the set $S$).
\end{lemma}
\begin{proof}
Let us define the real-analytic, $2\pi$-periodic function $f$ by the relation
$\lead(z)=|z|^mf(\theta)$, where we write $z=|z|e^{i\theta}$. Then, we compute 
\begin{equation}\label{E:zbar}
\partl{\lead}{\zbar}(e^{i\theta}) \ = \ 
        \frac{e^{i\theta}}{2}\left\{mf(\theta)+if^\prime(\theta)\right\}.
\end{equation}
We record two facts:
\begin{enumerate}
\item[a)] Since $f\in\smoo^\omega(\rl)$ and $2\pi$-periodic, 
$\text{\#}\{\theta\in[0,2\pi): f(\theta)=0\}$ is an even number.
\item[b)] Since $(\partial\lead/\partial\zbar)^{-1}\{0\}=\{0\}$, $f(\theta)$ and $f^\prime(\theta)$
cannot simultaneously vanish for any $\theta\in[0,2\pi)$.
\end{enumerate}
Thus, we have two
closed paths $\gamma_1,\gamma_2:[0,2\pi]\lrarw\cplx\setminus\{0\}$, defined by:
\begin{align}
\gamma_1(\theta) \ &:= \ mf(\theta)+if^\prime(\theta), \notag \\
\gamma_2(\theta) \ &:= \ e^{i\theta}. \notag
\end{align}
Recalling that the winding number is additive across products, we get:
\[
{\sf Wind}\left(\partl{\lead}{\zbar}(e^{i\bcdot}),0\right) \ = \ 
    {\sf Wind}(\gamma_1,0)+{\sf Wind}(\gamma_2,0).
\]
Hence, in view of the above and Lemma~\ref{L:windex}, it suffices for us to show that
\begin{equation}\label{E:toshow}
{\sf Wind}(\gamma_1,0) \ = \ -\frac{\text{\#}\{\theta\in[0,2\pi):f(\theta)=0\}}{2}.
\end{equation} 

Let us first consider the case when $f^{-1}\{0\}\neq\emptyset$. Without loss of generality, we may
assume that $f(0)=0$. Let
\[
0 = \theta_1 < \theta_1 < \dots\theta_N < 2\pi
\]
denote the distinct zeros of $\left.f\right|_{[0,2\pi)}$. Let $\phi:[0,2\pi]\lrarw\rl$ be a
function having the following properties (recall that by (b) above $f$ has only simple zeros):
\begin{itemize}
\item $\phi(\theta_j)=0, \ j=1,\dots,N$;
\item $\phi^\prime(\theta_j)f^\prime(\theta_j)>0, \ j=1,\dots,N$;
\item $|\phi^\prime(\theta_j)| > |f^\prime(\theta_j)|, \ j=1,\dots,N$;
\item $\left[\left.\phi\right|_{(\theta_{j-1},\theta_j)}\right]^\prime$ has {\em precisely} one
simple zero in $(\theta_{j-1},\theta_j), \ j=1,\dots,N$; and
\item $\phi$ has a $\smoo^\infty$-smooth periodic extension to $\rl$.
\end{itemize}
In view of the third property of $\phi$, there exists a constant $K>0$ such that
\begin{equation}\label{E:deriv}
|\phi^\prime(\theta_j)| \ > \ |f^\prime(\theta_j)| \ > \ K, \ j=1,\dots,N.
\end{equation}
Define the homotopy $H:[0,2\pi]\times[0,1]\lrarw\cplx$ by
\[
H(\theta,t) \ := \ m[(1-t)f(\theta)+t\phi(\theta)]+i[(1-t)f^\prime(\theta)+t\phi^\prime(\theta)].
\]
Note that, by construction 
\begin{align}
\er(H)(\theta,t)=0 \ \iff& \ f(\theta)=0, \notag \\
f(\theta)=0 \ \Longrightarrow& \ |(1-t)f^\prime(\theta)+t\phi^\prime(\theta)|>K. \notag
\end{align}
Hence, in fact, $H([0,2\pi]\times[0,1])\subset\cplx\setminus\{0\}$. Thus $\gamma_1$ is homotopic
in $\cplx\setminus\{0\}$ to the path $\varGamma_1:[0,2\pi]\lrarw\cplx\setminus\{0\}$ given by
\[
\varGamma_1(\theta) \ = \ m\phi(\theta)+i\phi^\prime(\theta), \;\; \theta\in[0,2\pi].
\]
By construction, the number of times that $\varGamma_1$ winds around the origin is 
half the number of times that $\varGamma_1$ intersects the real axis. But, since, by construction,
$\varGamma_1$ is oriented clockwise, we get, by homotopy invariance of the winding number:
\begin{equation}\label{E:almost1}
{\sf Wind}(\gamma_1,0) \ = \ {\sf Wind}(\varGamma_1,0)\ = \
                 -\frac{\text{\#}\{\theta\in[0,2\pi):f(\theta)=0\}}{2}.
\end{equation}
\smallskip

In the case when $f^{-1}\{0\}=\emptyset$, $\gamma_1$ never crosses the real axis. Hence
\begin{equation}\label{E:almost2}
f^{-1}\{0\}=\emptyset \ \Longrightarrow \
{\sf Wind}(\gamma_1,0) \ = \ 0.
\end{equation}
From \eqref{E:almost1} and \eqref{E:almost2} we see that \eqref{E:toshow} has been established.
This establishes our result.
\end{proof}
\smallskip

The last result in this section provides a Maslov-index calculation for the graph of a 
homogeneous polynomial $\lead$ that is, in contrast to Lemma~\ref{L:key}, {\em complex}-valued.
It will find no application later in this paper, but we present it as it might be of independent
interest.

\begin{lemma}\label{L:extra}
Let $\lead$ be a non-holomorphic, complex-valued polynomial that is homogeneous of 
degree $m$ and such that $(\partial \lead/\partial\zbar)^{-1}\{0\}=\{0\}$.
Define the polynomial $\derivF\in\cplx[z,w]$ by the relation
\[
\derivF(z,\zbar) \ = \ \partl{\lead}{\zbar}(z,\zbar)
\]
by making explicit the dependence of $\partial \lead/\partial\zbar$ on $z$ and $\zbar$.
Let $\dellead$ be the polynomial defined as $\dellead(z):=\derivF(z,1)$. Then,
\[
\MasInd(\Gamma(\lead),0) \ = \ 2\left(\sum\left\{\mu(\zt):\zt\in\dellead^{-1}\{0\}\bigcap\dee
                    \right\}\right)-(m-1),
\]
where $\mu(\zt)$ denotes the multiplicity of $\zt$ as a zero of the polynomial $\dellead$.
\end{lemma}
\begin{proof}
Note that, by hypothesis, the path $(\partial \lead/\partial\zbar)(e^{i\bcdot})$ does not
pass through the origin. Hence, in view of \eqref{E:windex}, we can explicitly compute
the desired winding number to get:
\begin{equation}\label{E:subst}
\MasInd(\Gamma(\lead),0) \ = \
\frac{1}{2\pi i}\int_0^{2\pi}\frac{\bdy^2_{z\zbar}\lead(e^{i\theta})ie^{i\theta}-
        \bdy^2_{\zbar\zbar}\lead(e^{i\theta})ie^{-i\theta}}{\bdy_{\zbar}\lead(e^{i\theta})}d\theta.
\end{equation}
We now compute that
\begin{align}
\bdy^2_{z\zbar}\lead(e^{i\theta})ie^{i\theta}-
\bdy^2_{\zbar\zbar}\lead(e^{i\theta})ie^{-i\theta} \ &= \
ie^{i\theta}\left[\frac{1}{z^{m-1}}\derivF(z^2,1)\right]^\prime_{z=e^{i\theta}}, \label{E:calc1} \\
\bdy_{\zbar}\lead(e^{i\theta}) \ &= \
\left.\frac{1}{z^{m-1}}\derivF(z^2,1)\right|_{z=e^{i\theta}}. \label{E:calc2}
\end{align}
From \eqref{E:subst}, \eqref{E:calc1} and \eqref{E:calc2}, we get
\[
\MasInd(\Gamma(\lead),0) \ = \
\frac{1}{2\pi i}\oint_{S^1}\frac{\left[\frac{1}{z^{m-1}}\dellead(z^2)\right]^\prime_{z=\zt}}
        {\frac{1}{\zt^{m-1}}\dellead(\zt^2)}d\zt.
\]
Since, by hypothesis, the denominator in the above integral never vanishes, the Argument Principle
gives us
\[
\MasInd(\Gamma(\lead),0) \ = \ 2\left(\sum\left\{\mu(\zt):\zt\in\dellead^{-1}\{0\}\bigcap\dee
                \right\}\right)-(m-1).
\]
\end{proof}
\medskip

\section{The proof of Theorem~\ref{T:posInd}}\label{S:posInd}

We introduce some notations that will be needed in the proof of Theorem~\ref{T:posInd}.
First, we define the Banach space $\smoo^\alpha(\bdy\dee;\mathbb{F}), \ \alpha\in(0,1)$, where 
$\mathbb{F}$ will stand for either $\rl$ or $\cplx$ in the following proof, as
\[
\smoo^\alpha(\bdy\dee;\mathbb{F}):=\left\{f:\bdy\dee\lrarw\mathbb{F}:\sup_{\theta\in\rl}|f(\cis{\theta})|
+\sup_{\theta\neq\phi\in\rl}\frac{|f(\cis{\theta})-f(\cis{\phi})|}{|\theta-\phi|^\alpha}<\infty\right\},
\]
where the norm on this Banach space is:
\[
\hnrm{f} \ := \
\sup_{\theta\in\rl}|f(\cis{\theta})|
+\sup_{\theta\neq\phi\in\rl}\frac{|f(\cis{\theta})-f(\cis{\phi})|}{|\theta-\phi|^\alpha}.
\]
We will also have occasion to use the following abbreviation
\[
\hlpart{f} \ := \ 
\sup_{\theta\neq\phi\in\rl}\frac{|f(\cis{\theta})-f(\cis{\phi})|}{|\theta-\phi|^\alpha}.
\]
In what follows, $A(\bdy\dee)$ will denote the class of restrictions to the unit circle of
functions that are holomorphic on $\dee$ and continuous on $\overline{\dee}$. 
For any $f\in\smoo(\bdy\dee;\mathbb{F})$ we will denote the Fourier series of $f$ as follows:
\[
f \ \thicksim \ \sum_{n\in\zahl}\widehat{f}(n)\cis{n\theta}.
\]
It is well known that if $f\in\smoo^\alpha(\bdy\dee;\mathbb{F})$ with $\alpha\in(0,1)$,
then any harmonic conjugate on $\dee$ of the Poisson integral of $f$ extends to a function
on $\overline{\dee}$ and its restriction to $\bdy\dee$, say $h_f$, is of class
$\smoo^\alpha(\bdy\dee;\mathbb{F})$. In this paper, $\cnj[f]$ will denote that $h_f$
which satisfies (in our Fourier-series notation) $\widehat{h_f}(0)=0$.
In terms of Fourier series:  
\[
\cnj[f] \ \thicksim \
\sum_{n\in\zahl}-i \ {\rm sgn}(n)\widehat{f}(n)\cis{n\theta}.
\]
We call $\cnj[f]$ the {\em conjugate of $f$}. Recall that the operator 
$\cnj:\smoo^\alpha(\bdy\dee;\mathbb{F})\lrarw\smoo^\alpha(\bdy\dee;\mathbb{F})$
is a certain singular-integral operator that is bounded on $\smoo^\alpha(\bdy\dee;\mathbb{F})$.
We shall use this fact (which we assume the reader is familiar with) in Step~1 of our proof below.
\medskip

\noindent{{\bf The proof of Theorem~\ref{T:posInd}.} Let $(\srf,p)$ be as stated in the hypothesis 
of the theorem. By definition, there is a neighbourhood $U_p$ of $p$ and
holomorphic coordinates $(z,w)$ centered at $p$ such 
that $\srf$ is locally defined by
\begin{equation}\label{E:normform2*}
\srf\cap U_p: \quad w \ = \ \lead(z) + \rem(z) \quad(\text{for $|z|$ small}),
\end{equation}
where $\lead$ is a real-valued polynomial that is homogeneous of degree $m$, and
$\rem(z)=O(|z|^{m+1})$. From this last fact, and from \eqref{E:windex} in 
Lemma~\ref{L:windex}, we see that
\[
\MasInd(\Gamma(\lead+\rem),0) \ = \ \MasInd(\Gamma(\lead),0).
\]
This is seen by considering the relevant winding numbers of small circles centered at $z=0$.
Now note that the index $\MasInd(\srf,p)$ is, by construction, invariant under holomorphic 
changes of coordinate. Hence
\[
\MasInd(\Gamma(\lead),0) \ = \ \MasInd(\Gamma(\lead+\rem),0) \ = \ \MasInd(\srf,p) \ > \ 0.
\]
Applying \eqref{E:keyWindex} to the above statement, we may conclude, without loss of 
generality, that
\begin{equation}\label{E:posi}
\lead(z) \ > \ 0 \;\; \forall z\in\cplx\setminus\{0\}.
\end{equation}}

Let us define
\[
\rho \ := \ \sup\{s>0:\Gamma(\lead+\rem; \overline{D(0;s)})\subset U_p, \ \text{and} \ 
            \overline{D(0;s)}\times\{0\}\subset U_p\},
\]
(here, and elsewhere in this paper, $\Gamma(\lead+\rem; \overline{D(0;s)})$ denotes the 
portion of the graph of $(\lead+\rem)$ over the closed disc $\overline{D(0;s)}$). In the remainder of this
proof, {\em whenever we use the parameter $r>0$, we will assume that $0<r<3\rho/4$}. In view of
\eqref{E:posi}, and the fact that $\lead$ is homogeneous, $\lead^{-1}\{1\}$ is a real-analytic curve
that meets each ray originating at $0$ at precisely one point. To see this, we first note that, by
homogeneity, for each fixed $\tht\in [0,2\pi)$, $\lead(r\cis{\tht})=C_\tht r^m \ \forall r>0$. By
\eqref{E:posi}, $C_\tht>0$, whence the ray $\{r\cis{\tht}:r>0\}$ intersects $\lead^{-1}\{1\}$ precisely
at $\cis{\tht}/C_\tht^{1/m}$. It now follows from basic topology that $\lead^{-1}\{1\}$ is a   
simple closed curve that encloses $0$. Thus, we can define
\begin{align}
\dom \ &:= \ \text{the region in $\cplx$ enclosed by $\lead^{-1}\{1\}$}, \notag \\
G \ &:= \ \text{the unique Riemann mapping of $\dee$ onto $\dom$ such that $G(0)=0$, $G^\prime(0)>0$}. \notag
\end{align}
Note that as $\lead^{-1}\{1\}$ is a real-analytic, simple closed curve:
\begin{itemize}
\item $G$ extends to a homeomorphism between $\overline{\dee}$ and $\overline{\dom}$ such that
$G:(\dee,\bdy\dee)\lrarw (\dom,\bdy\dom)$; and
\item By the Schwarz Reflection Principle, $\exists\eps>0$ such that $G$ extends to a function 
$\widetilde{G}\in\hol(D(0;1+\eps))$ 
such that $\widetilde{G}^\prime(\zt)\neq 0 \ \forall\zt\in\bdy\dee$.
\end{itemize}
Let us define $g:=\left.\widetilde{G}\right|_{\bdy\dee}$ and $g_r:=\kappa rg$, where 
$\kappa:=(1/2)\left[\sup_{\zt\in\bdy\dee}|g(\zt)|\right]^{-1}$.  
\smallskip

\noindent{{\bf Step~1.} {\em Constructing the relevant Bishop's Equation}}

\noindent{Let us fix an $\alpha\in(0,1)$. Define the mapping 
$\bish:\smoo^\alpha(\bdy\dee;\rl)\lrarw\smoo^\alpha(\bdy\dee;\cplx)\cap A(\bdy\dee)$ by
\[
\bish[\psi] \ := \ \psi+i\cnj[\psi].
\]  
Recall that $\cnj[\psi]$ denotes the conjugate of $\psi$}. It is well-known that for each
$\alpha\in(0,1)$, there exists a $\gamma_\alpha>0$ such that
\begin{equation}\label{E:hlreg}
\hnrm{\cnj[\psi]} \ \leq \ \gamma_\alpha\hnrm{\psi} \ \forall\psi\in\smoo^\alpha(\bdy\dee;\rl).
\end{equation}
We remark here that
the optimal dependence of $\gamma_\alpha$ on the parameter $\alpha\in (0,1)$ is known; see 
\cite[Part~IV, Theorem~2.4]{merkerPorten:heCRf06}. Specifically, if $\psi$ is in the 
H{\"o}lder class $\smoo^{k,\alpha}(\bdy\dee;\mathbb{F})$ then
\[
 \|\cnj[\psi]\|_{\smoo^{k,\alpha}} \ \leq \ C(k)\tfrac{1}{\alpha(1-\alpha)}\|\psi\|_{\smoo^{k,\alpha}},
\]
where $C(k)>0$ denotes a constant that depends only on $k$. However, we shall not require this degree
of precision in the arguments that follow and we shall work with $\gamma_\alpha>0$.
Define the open set $\OM_\alpha\subset\smoo^{\alpha}(\bdy\dee;\rl)$ by
\[
\OM_\alpha \ := \ \{\psi\in\smoo^{\alpha}(\bdy\dee;\rl):\sqrt{1+\gamma^2_\alpha}\hnrm{\psi}<3\rho/8\}.
\]
Finally, define the function $\Phi:\OM_\alpha\times(0,3\rho/4)\lrarw\smoo^{\alpha}(\bdy\dee;\rl)$ by
\begin{align}
\Phi(\psi,r) :=& -(\kappa r)^m
    +\lead\circ(g_r+\cis{\bcdot}\bish[\psi])+(\er\rem)\circ(g_r+\cis{\bcdot}\bish[\psi]) \notag \\
    &\qquad\quad +\cnj[(\mi\rem)\circ(g_r+\cis{\bcdot}\bish[\psi])] \notag \\
=& \bdy_z\lead(g_r)\cis{\bcdot}\bish[\psi]+\bdy_{\zbar}\lead(g_r)\overline{\cis{\bcdot}\bish[\psi]}
    +Q(g_r,\cis{\bcdot}\bish[\psi]) \notag \\
    &\qquad\quad+(\er\rem)\circ(g_r+\cis{\bcdot}\bish[\psi])
    +\cnj[(\mi\rem)\circ(g_r+\cis{\bcdot}\bish[\psi])] \label{E:simple}
\end{align}
where we define
\[
Q(X,Y) \ := \ 
\sum_{j=2}^m\sum_{\mu+\nu=j}\frac{1}{\mu!\nu!}\bdy^\mu_z\bdy^\nu_{\zbar}\lead(X)Y^\mu\overline{Y}^\nu.
\]
We are now in a position to assert the following:}
\smallskip

\noindent{{\bf Fact~A.} {\em If, for some $(\psi_0,r^0)\in\OM_\alpha\times(0,3\rho/4)$, 
$\Phi(\psi_0,r^0)=0$, then there is an analytic disc $F\in\hol(\dee;\CC)\cap\smoo^\alpha(\overline{\dee})$,
which is a small perturbation of the analytic disc $(g_{r^0},(\kappa r^0)^m)$, such that 
$F(\bdy\dee)\subset\srf$.}}

\noindent{To justify the above assertion, note that as $\Phi(\psi_0,r^0)$ is identically zero,
\[
\Phi(\psi_0,r^0)+i\bish[(\mi\rem)\circ(g_{r^0}+\cis{\bcdot}\bish[\psi_0])]
\]
is the boundary value of a holomorphic function. However
\begin{multline}\label{E:bdyFit}
\Phi(\psi_0,r^0)+i\bish[(\mi\rem)\circ(g_{r^0}+\cis{\bcdot}\bish[\psi_0])] \\
= -(\kappa r^0)^m
    +\lead\circ(g_{r^0}+\cis{\bcdot}\bish[\psi_0])+\rem\circ(g_{r^0}+\cis{\bcdot}\bish[\psi_0]).
\end{multline}
Clearly, the Poisson integral of the function
\[
(g_{r^0},(\kappa r^0)^m)+
\left(\cis{\bcdot}\bish[\psi_0],\,i\bish[(\mi\rem)\circ(\cis{\bcdot}\bish[\psi_0]+g_{r^0})]\right)
\]
is an analytic disc $F:=(F_1,F_2)$, and by \eqref{E:bdyFit}
\[
F_2(\zt) \ = \ \lead\circ F_1(\zt) + \rem\circ F_1(\zt) \;\; \forall\zeta\in\bdy\dee,
\]
which is precisely the fact asserted above.}
\smallskip

To show that the analytic discs described in Theorem~\ref{T:posInd} vary smoothly with respect to the
parameter $r$, we have to establish that each of these discs exists. To this end, the above discussion
helps in setting the following
\smallskip

\noindent{{\bf Intermediate Goal.} {\em To solve the equation $\Phi(\psi,r)=0$ for all sufficiently
small values of the parameter $r$.}}
\smallskip

\noindent{{\bf Step~2.} {\em Setting up an equivalent equation to the functional equation $\Phi(\psi,r)=0$}}

\noindent{Consider the linear operator (which is bounded from $\smoo^\alpha(\bdy\dee;\rl)$ to 
$\smoo^\alpha(\bdy\dee;\rl)$ owing to \eqref{E:hlreg} above)
\begin{align}
\Lambda_r:\psi\longmapsto & \
\bdy_z\lead(g_r)\cis{\bcdot}\bish[\psi]+\bdy_{\zbar}\lead(g_r)\overline{\cis{\bcdot}\bish[\psi]} \notag \\
    =& \ 2\er\left\{\bdy_z\lead(g_r)\cis{\bcdot}\bish[\psi]\right\}. \notag
\end{align}}

\noindent{{\bf {\em Note.}} Before we engage in technicalities, we ought to point out the difficulties
inherent in this problem. First note that:
\[
\text{\em The Fr{\'e}chet (partial)
derivative $\left.\partial_\psi\Phi\right|_{(\psi,0)}$ is not invertible for any $\psi\in\OM_\alpha$.}
\]
Suppose that could show that the Fr{\'e}chet derivative 
$\left.\partial_\psi\Phi\right|_{(\psi^0,r^0)}$ is invertible for some 
$(\psi^0,r^0)\in\OM_\alpha\times(0,3\rho/4)={\sf Dom}(\Phi)$. With this, {\em we would still be
unable to invoke the Implicit Function Theorem to either assert the existence of analytic discs 
attached to $\srf$ or to infer their smooth dependence on $r$ in a neighbourhood of $r_0$.} This
is because it must first be established that
$\Phi(\psi^0,r^0)=0$! This is precisely our Intermediate Goal above. The inquisitive reader is 
directed also to Remark~\ref{R:Forst} below.}
\smallskip

\noindent{{\bf Claim.} {\em $\Lambda_r$ is an isomorphism.}}

\noindent{To show that $\Lambda_r$ is surjective, note that it suffices to show that given any
$f\in\smoo^{\alpha}(\bdy\dee;\rl)$, there exists a function
$a_f\in A_0^\alpha(\bdy\dee)$, where
\[
A_0^\alpha(\bdy\dee) \ := \ \{h\in\smoo^{\alpha}(\bdy\dee;\cplx):\widehat{h}(0)\in\rl, \ 
            \text{and} \ \widehat{h}(j)=0 \ \forall j\leq-1\},
\]
such that $2\er\left\{\bdy_z\lead(g_r)\cis{\bcdot}a_f\right\}=f$. Note that, from the discussion
preceding Step~2, it can be inferred that
\[
\lead\circ(r\kappa\widetilde{G}) - (\kappa r)^m \ \text{\em vanishes on $\bdy\dee$}.
\]
Thus, recalling that $\widetilde{G}^\prime(\zt)\neq 0 \ \forall\zt\in\bdy\dee$, there exists a 
$\delta>0$ and a function $R\in\smoo^\omega({\rm Ann}(0;1-\delta,1+\delta))$ such that (we
treat $r$ as a parameter here)
\begin{itemize}
\item $R(z)>0 \ \forall z\in{\rm Ann}(0;1-\delta,1+\delta)$; and
\item $\lead\circ(r\kappa\widetilde{G})-(\kappa r)^m=r^mR(z)(|z|^2-1) \ 
\forall z\in{\rm Ann}(0;1-\delta,1+\delta)$.
\end{itemize}
By the chain rule (recall that $g_r$ is the restriction of a holomorphic function):
\begin{align}
\cis{\bcdot}(\bdy_z\lead)\circ(g_r) \ &= \ 
\left.\partl{(\lead\circ(\kappa r\widetilde{G})-(\kappa r)^m)}{z}\right|_{\bdy\dee}\times
\frac{\cis{\bcdot}}{\kappa r\widetilde{G}^\prime} \label{E:trick} \\
&= \ r^m\left(\left.\zbar R\right|_{\bdy\dee}\right)\frac{\cis{\bcdot}}{\kappa r\widetilde{G}^\prime}.
\end{align}
The second equality follows from the fact that $(|z|^2-1)\bdy_zR(z)$ vanishes on $\bdy\dee$. Hence,
the desired $a_f$ is a solution to the equation
\begin{equation}\label{E:matching}
2\frac{r^{m-1}R(\cis{\bcdot})}{\kappa}\er\left(\frac{a_f}{\widetilde{G}^\prime}\right) \ = \ f \;\;
\text{\em with $a_f\in A_0^\alpha(\bdy\dee)$}.
\end{equation}
It was shown by Privalov that -- owing to the normalisation 
condition that $a_f$ belong to $A_0^\alpha(\bdy\dee)$ -- the equation \eqref{E:matching}
has a {\em unique} solution in $A_0^\alpha(\bdy\dee)$ given by
\[
a_f(\zt) \ = \ 
\frac{\kappa\widetilde{G}^\prime(\zt)}{2r^{m-1}} \ \bish\left[\frac{f}{R(\cis{\bcdot})}\right](\zt) \;\;
\forall\zt\in\bdy\dee.
\]
This establishes that $\Lambda_r$ is surjective, and the uniqueness of $a_f$ establishes that
it is injective. Hence the claim.}
\smallskip

To complete the discussion on the invertibility of $\Lambda_r$ we note that 
$\Lambda_r^{-1}=\invrt_r$, where
\[
\invrt_r[f] \ = \ 
\er\left\{\frac{\kappa\widetilde{G}^\prime(\cis{\bcdot})}{2r^{m-1}} \ 
\bish\left[\frac{f}{R(\cis{\bcdot})}\right]
\right\}.
\]
Furthermore, from the fact that $\bish=\mathbb{I}_{\smoo^\alpha}+i\cnj$, and from the estimate
\eqref{E:hlreg}, we get the following important estimate: there exists a $K_\alpha>0$ such that
\begin{equation}\label{E:invrtreg}
\hnrm{\invrt_r[f]} \ \leq \ K_\alpha r^{1-m}\hnrm{f} \ \forall f\in\smoo^\alpha(\bdy\dee;\rl).
\end{equation} 

Finally, by applying $\invrt_r$ to the equation \eqref{E:simple}, we see that solving the equation
\[
\Phi(\psi,r) \ = \ 0, \;\; (\psi,r)\in\OM_\alpha\times(0,3\rho/4)
\]
is equivalent to solving
\begin{align}
\psi+&\invrt_r\left[Q(g_r,\cis{\bcdot}\bish[\psi])+(\er\rem)\circ(\cis{\bcdot}\bish[\psi]+g_r)
    +\cnj[(\mi\rem)\circ(\cis{\bcdot}\bish[\psi]+g_r)]\right] \notag \\ 
    &\equiv \ \psi-H(\psi;r) \notag \\
    &= \ 0. \notag
\end{align}  
In view of this, the goal presented at the end of Step~1 is modified as follows:
\smallskip

\noindent{{\bf Modified Intermediate Goal.} {\em To find a fixed point of the map 
$H(\bcdot;r):\OM_\alpha\lrarw\smoo^\alpha(\bdy\dee;\rl)$ for each sufficiently small     
value of the parameter $r$.}}
\medskip

\noindent{{\bf Step~3.} {\em Some estimates}}

\noindent{We shall use the contraction mapping principle to establish the modified goal above. 
For this purpose, we will (for a fixed $r>0$) determine the image under $H(\bcdot;r)$
of a small closed ball in $\smoo^\alpha(\bdy\dee;\rl)$ centered at $0$. We will also show that
$H(\bcdot;r)$ is a contraction on this ball. This requires some estimates.
\smallskip

Since $\rem(z)=O(|z|^{m+1})$, it follows that there is a large positive 
constant $L>0$ that is independent of $r>0$ such that
\begin{equation}\label{E:estReRem_u}
\uninrm{(\er\rem)\circ g_r} \ \leq \ L\left(\frac{r}{\rho}\right)^{m+1},
\end{equation}
and
\begin{equation}\label{E:estReRem_hlpart}
\hlpart{(\er\rem)\circ g_r} \ \leq \ L\left(\frac{r}{\rho}\right)^{m+1}\|\er\rem\|_{\smoo^1}
    \kappa\rho\sup_{\theta\neq\phi\in\rl}
    \frac{|g(\cis{\theta})-g(\cis{\phi})|}{|\theta-\phi|^\alpha}.
\end{equation}
\smallskip

We now set the stage for showing that for each $r>0$ sufficiently small, 
$H(\bcdot;r)$ is a contraction on the closed ball 
$\overline{\mathbb{B}_{\smoo^\alpha}(0;r^{1+\delta})}$, where 
we pick and fix $\delta\in(1/2,1)$. Furthermore, we shall work with $r\in(0,r_1)$, 
where $r_1>0$ is so small that 
$r/(100\sqrt{1+\gamma^2_\alpha})\geq r^{1+\delta} \ \forall r\in(0,r_1)$.
This will ensure that all values of $\psi$ under consideration satisfy the constraint
\begin{equation}\label{E:constraint}
\hnrm{\psi} \ \leq \ \frac{r}{100\sqrt{1+\gamma^2_\alpha}}.
\end{equation}
In the next few estimates, {\em we shall assume that these constraints are in effect even if
not explicitly stated.} To simplify notation we set
$A(\mu,\nu,\psi):=\bish[\psi]^\mu\overline{\bish[\psi}]^\nu$.
We first estimate:
\begin{align}
&\uninrm{Q(g_r,\cis{\bcdot}\bish[\psi_1])-Q(g_r,\cis{\bcdot}\bish[\psi_2])} \notag \\
&\quad \leq \ \sum_{j=2}^m\sum_{\mu+\nu=j}\frac{r^{m-j}}{\mu!\nu!}\uninrm{
\bdy^\mu_z\bdy^\nu_{\zbar}\lead(\kappa g)\left(A(\mu,\nu,\psi_1)-A(\mu,\nu,\psi_2)\right)}
\notag \\
&\quad \leq \ \sum_{j=2}^m\sum_{\mu+\nu=j}\frac{r^{m-j}}{\mu!\nu!}
\sup_{\zt\in\overline{\dee}}\left|\bdy^\mu_z\bdy^\nu_{\zbar}\lead(\zt)\right| \notag \\
&\qquad\qquad\qquad\quad \ \times\sqrt{1+\gamma^2_\alpha}^{j-1}\uninrm{\bish[\psi_1-\psi_2]}
\left(\mu r^{(1+\delta)(j-1)}+\nu r^{(1+\delta)(j-1)}\right) \notag \\
&\quad \leq \ C_\alpha\sum_{j=2}^m jr^{(m-1)+\delta(j-1)}\hnrm{\psi_1-\psi_2}. \label{E:estQu-diff}
\end{align}}
It is the bound $\hnrm{\psi_j}\leq r^{(1+\delta)}, \ j=1,2$, that leads to the second
inequality above. Next, we estimate, using the fundamental theorem of calculus:
\begin{align}
&\uninrm{(\er\rem)\circ(g_r+\cis{\bcdot}\bish[\psi_1])-(\er\rem)\circ(g_r+\cis{\bcdot}\bish[\psi_2])} 
\notag \\
&\leq \ 2\sup_{\theta\in\rl}\left|\er\left\{\int_0^1\bdy_z(\er\rem)(g_r(\cis{\theta})+\cis{\theta}
    (t\bish[\psi_1]+(1-t)\bish[\psi_2])(\cis{\theta}))dt\right\}\right|
    \uninrm{\bish[\psi_1-\psi_2]}. \notag
\end{align}
Since $\rem(z)=O(|z|^{m+1})$, and since the constraint \eqref{E:constraint} ensures that
\[ 
{\rm range}(g_r+t\cis{\bcdot}\bish[\psi_1]+(1-t)\cis{\bcdot}\bish[\psi_2]) \ \Subset 
\ {\sf dil}_r[{\rm domain}(\rem)] \;\; \forall t\in[0,1]
\] 
(where ${\sf dil}_r$ denotes the dilation on $\cplx$ by a factor of $r$), there is a large positive 
constant $L>0$ that is independent of $r>0$ such that
\begin{multline}\label{E:estReRemu-diff}
\uninrm{(\er\rem)\circ(g_r+\cis{\bcdot}\bish[\psi_1])-(\er\rem)\circ(g_r+\cis{\bcdot}\bish[\psi_2])} \\
\leq \ L\sqrt{1+\gamma^2_\alpha}\hnrm{\psi_1-\psi_2}\left(\frac{r}{\rho}\right)^{m}.
\end{multline}
\smallskip

To simplify the presentation of our next estimate, let us set
\[
C(j,\alpha) \ := \ \hnrm{\cis{j\bcdot}}, \;\; j\in\zahl, \quad
M_\alpha \ := \ \max\left(\hnrm{\bish[\psi_1]},\hnrm{\bish[\psi_2]}\right),
\]
and write $j=\mu+\nu$, $2\leq j\leq m$. Note that:
\begin{align}
&\bdy^\mu_z\bdy^\nu_{\zbar}\lead(g_r)A(\mu,\nu,\psi_1)-
        \bdy^\mu_z\bdy^\nu_{\zbar}\lead(g_r)A(\mu,\nu,\psi_2) \notag \\
&= \ r^{m-j}\bdy^\mu_z\bdy^\nu_{\zbar}\lead(\kappa g)\cis{(\mu-\nu)\bcdot} \notag \\
&\quad\times\left( \ \overline{\bish[\psi_1-}\psi_2]\right.
    \sum_{t=0}^{\nu-1}\overline{\bish[\psi_1}]^tA(\mu,\nu-t-1,\psi_2)
    +\bish[\psi_1-\psi_2]
    \sum_{s=0}^{\mu-1}\left. A(s,\nu,\psi_1)\bish[\psi_2]^{\mu-s-1}\right). \notag
\end{align}
Then, it is easy to see that
\begin{align}
&\hlpart{\bdy^\mu_z\bdy^\nu_{\zbar}\lead(g_r)A(\mu,\nu,\psi_1)-
    \bdy^\mu_z\bdy^\nu_{\zbar}\lead(g_r)A(\mu,\nu,\psi_2)} \notag \\
&\leq \ C_\alpha r^{m-j}\smoonrm{\lead}{j+1}\kappa\sup_{\theta\neq\phi\in\rl}
\frac{|g(\cis{\theta})-g(\cis{\phi})|}{|\theta-\phi|^\alpha}jM_\alpha^{j-1}
\uninrm{\bish[\psi_1-\psi_2]}\notag \\
&\quad +C_\alpha r^{m-j}
    \smoonrm{\lead}{j}jM_\alpha^{j-1}\left(\uninrm{\bish[\psi_1-\psi_2]}(C(\mu-\nu,\alpha)+(j-1))
    +\hnrm{\bish[\psi_1-\psi_2]}\right) \notag
\end{align}
From this estimate, and the fact that $0\leq M_\alpha\leq \sqrt{1+\gamma^2_\alpha}r^{1+\delta}$,
we conclude that there exists a constant $C_\alpha>0$, depending only on $\alpha$, such that
\begin{equation}\label{E:estQhlpart-diff}
\hlpart{Q(g_r,\cis{\bcdot}\bish[\psi_1])-Q(g_r,\cis{\bcdot}\bish[\psi_2])} \ \leq
\ C_\alpha\sum_{j=2}^m r^{(m-1)+\delta(j-1)}\hnrm{\psi_1-\psi_2}.  
\end{equation}
\smallskip

Finally, using exactly the same technique that led to the estimate \eqref{E:estReRemu-diff},
we compute:
\begin{align}
&\hlpart{(\er\rem)\circ(g_r+\cis{\bcdot}\bish[\psi_1])-(\er\rem)\circ(g_r+\cis{\bcdot}\bish[\psi_2])} 
\notag \\
&\leq \ 2\sup_{\theta\in\rl}\left|\er\left\{\int_0^1\bdy_z(\er\rem)(g_r(\cis{\theta})+\cis{\theta}
    (t\bish[\psi_1]+(1-t)\bish[\psi_2])(\cis{\theta}))dt\right\}\right|
    \hnrm{\bish[\psi_1-\psi_2]} \notag \\
&\quad\qquad + L\left(\frac{r}{\rho}\right)^{m}\uninrm{\bish[\psi_1-\psi_2]}
    \int_0^1\rho\hlpart{\kappa g +\frac{\cis{\bcdot}(t\bish[\psi_1]+(1-t)\bish[\psi_2])}{r}}dt. \notag
\end{align}
Arguing in an analogous manner as above, we conclude that there exists a large uniform constant
$L>0$ and a $C_\alpha>0$, depending only on $\alpha$, such that: 
\begin{multline}\label{E:estReRemhlpart-diff}
\hlpart{(\er\rem)\circ(g_r+\cis{\bcdot}\bish[\psi_1])-(\er\rem)\circ(g_r+\cis{\bcdot}\bish[\psi_2])} \\
\leq \ 2L\left(\frac{r}{\rho}\right)^{m}\left(\sqrt{1+\gamma^2_\alpha}+C_\alpha r^{\delta}\right)
\hnrm{\psi_1-\psi_2}.
\end{multline}
\smallskip

We are now in a position to write down three key estimates that we need. In each of the
three estimates, there exists a constant $C_\alpha>0$ that depends only on $\alpha$ such that
the following inequalities hold. Firstly, from \eqref{E:estQu-diff}
and \eqref{E:estQhlpart-diff} we get
\begin{equation}\label{E:estQ}
\hnrm{Q(g_r,\cis{\bcdot}\bish[\psi_1])-Q(g_r,\cis{\bcdot}\bish[\psi_2])} \leq \ 
C_\alpha\sum_{j=2}^mr^{(m-1)+\delta(j-1)}\hnrm{\psi_1-\psi_2}.
\end{equation}
Next, from \eqref{E:estReRemu-diff} and \eqref{E:estReRemhlpart-diff}, we get 
\begin{equation}\label{E:estReRem}
\hnrm{(\er\rem)\circ(g_r+\cis{\bcdot}\bish[\psi_1])-(\er\rem)\circ(g_r+\cis{\bcdot}\bish[\psi_2])} \ 
\leq \ C_\alpha(1+r^\delta)r^m\hnrm{\psi_1-\psi_2}.
\end{equation}
Finally, note that the same arguments that
lead to \eqref{E:estReRemu-diff} and \eqref{E:estReRemhlpart-diff} also yield {\em exactly} analogous
estimates for $(\mi\rem)\circ(g_r+\cis{\bcdot}\bish[\psi])$. This observation, coupled with the bound
\eqref{E:hlreg} for the operator $\cnj$ gives us   
\begin{multline}\label{E:estImRem}
\hnrm{\cnj\left[(\mi\rem)\circ(g_r+\cis{\bcdot}\bish[\psi_1])-
        (\mi\rem)\circ(g_r+\cis{\bcdot}\bish[\psi_2])\right]} \\ 
\leq \ C_\alpha(1+r^\delta)r^m\hnrm{\psi_1-\psi_2}.
\end{multline}
All these estimates hold for $\psi_1,\psi_2\in\overline{\mathbb{B}_{\smoo^\alpha}(0;r^{1+\delta})}$.
\smallskip

\noindent{{\bf Step~4.} {\em Completing the proof}}

\noindent{Applying the bounds \eqref{E:invrtreg} for the operator $\invrt_r$ to the 
estimates \eqref{E:estQ}, \eqref{E:estReRem} and
\eqref{E:estImRem}, we see that there exists a constant $L_\alpha>0$ such that
\begin{align}
\hnrm{H(\psi_1;r)-H(\psi_2;r)} \ \leq& \ \frac{L_\alpha}{r^{m-1}}
        \left(r^m(1+r^\delta)+\sum_{j=2}^mr^{(m-1)+\delta(j-1)}\right)\hnrm{\psi_1-\psi_2} \notag \\ 
\leq& \ 2L_\alpha r^\delta\hnrm{\psi_1-\psi_2} \notag \\
& \ \forall\psi_1,\psi_2\in\overline{\mathbb{B}_{\smoo^\alpha}(0;r^{1+\delta})} \ \text{\em and} \ 
    \forall r\in(0,r_2), \label{E:cntract}
\end{align}
where $r_2\in(0,r_1)$ is so small that the second inequality is valid for all $r\in(0,r_2)$.
Furthermore, we deduce from the estimates \eqref{E:estReRem_u} and \eqref{E:estReRem_hlpart}
(and by the same argument that leads to the estimate \eqref{E:estImRem}) that there
exists a constant $K_\alpha>0$ such that
\begin{equation}\label{E:zeroImage}
\hnrm{H(0;r)} \ = \ \hnrm{\invrt_r\left[(\er\rem)\circ g_r+\cnj[(\mi\rem)\circ g_r]\right]} \
\leq \ K_\alpha r^2 \;\; \forall r\in (0,3\rho/4).
\end{equation}}

Let $r_3>0$ be so small that:
\begin{align}
2L_\alpha r^\delta \ &\leq \ 1/2 \quad \text{\em and} \notag \\
K_\alpha r^2 \ &\leq \ r^{1+\delta}/2 \;\; \forall r\in(0,r_3). \notag
\end{align}
Set $R_0 \ := \ \min(3\rho/4,r_2,r_3)$. Then, for any $r\in(0,R_0)$,
\begin{multline}\label{E:cntrctn}
\psi_1,\psi_2\in\overline{\mathbb{B}_{\smoo^\alpha}(0;r^{1+\delta})} \\
\Longrightarrow
\ \hnrm{H(\psi_1;r)-H(\psi_2;r)}\leq (1/2)\hnrm{\psi_1-\psi_2} 
\quad\text{[{\em due to} \eqref{E:cntract}]},
\end{multline}
and, furthermore
\begin{align}
\psi\in\overline{\mathbb{B}_{\smoo^\alpha}(0;r^{1+\delta})} \ \Longrightarrow
\ \hnrm{H(\psi;r)} &\leq (1/2)\hnrm{\psi}+\hnrm{H(0;r)} &&\text{[{\em due to} \eqref{E:cntrctn}]} \notag \\
&\leq r^{1+\delta}. &&\text{[{\em due to} \eqref{E:zeroImage}]} \notag
\end{align}
This last fact and the estimate \eqref{E:cntrctn} enable us to apply the contraction mapping principle
to $H(\bcdot;r):\overline{\mathbb{B}_{\smoo^\alpha}(0;r^{1+\delta})}\lrarw
\overline{\mathbb{B}_{\smoo^\alpha}(0;r^{1+\delta})}$ for each $r\in(0,R_0)$ --- owing to which
we get:   
\smallskip

\noindent{{\bf Fact~B.} {\em For each $r\in(0,R_0)$, there exists a unique
$\psi_r\in\overline{\mathbb{B}_{\smoo^\alpha}(0;r^{1+\delta})}$ such that \linebreak
$H(\psi_r;r)=\psi_r$.}}
\smallskip

Before proceeding any further, we record that (shrinking $R_0>0$ further if necessary) 
{\em $\hnrm{\psi_r}$ is not comparable to $\hnrm{g_r}(\thickapprox r) \ \forall r\in(0,R_0)$, which
ensures that the desired analytic discs will be non-constant.} Let us now write $G(r):=\psi_r$.
Recalling the discussions at the end of Step~1 and Step~2 of this proof, we see that the 
desired analytic discs $\mathfrak{g}(r)$ are the analytic maps defined by the boundary condition:
\[
\left.\mathfrak{g}(r)\right|_{\bdy\dee} \ = \ (g_r,(\kappa r)^m)+
\left(\cis{\bcdot}\bish[G(r)],\,i\bish[(\mi\rem)\circ(\cis{\bcdot}\bish[G(r)]+g_{r})]\right).
\]
The analytic discs {\em per se} are the Poisson integrals of the functions on the right-hand
side above. Standard facts about the Poisson integral imply that in order to show that
$\mathfrak{g}:(0,R_0)\lrarw A^\alpha(\dee;\cplx^2)$ is of class $\smoo^1$, it suffices
to show that $G$ is smooth on $(0,R_0)$.
\smallskip

We remark that the inequalities in Step~3, which culminate in the
estimate \eqref{E:cntract}, could have been carried out in an exactly analogous manner (i.e.
by applying the fundamental theorem of calculus appropriately) with both $\psi$ and $r$ taken 
to be variable. We refrained from doing this so as to avoid writing out lengthy, but essentially 
basic, estimates. However, the estimates in Step~3 have been presented sufficiently carefully that
we may leave it to the reader to emulate them, and verify that there exist constants 
$K_1, K_2>0$ such that, shrinking $R_0>0$ further if necessary:
\begin{align}
\hnrm{H(\psi_1;r_1)-H(\psi_2;r_2)} \ \leq \ (1/2)&\hnrm{\psi_1-\psi_2} + \left(K_1+
			\frac{K_2}{r_1^{m-1}r_2^{m-1}}\right)|r_1-r_2| \notag \\
\forall(\psi_1,r_1), (\psi_2,r_2)\in
		\{(\psi,r)&\in\smoo^\alpha(\bdy\dee;\rl)\times(0,R_0):\hnrm{\psi}\leq r^{1+\delta}\},
\label{E:conti}
\end{align}
where $\delta>0$ is as chosen in Step~3.
\smallskip

Let us now consider again the map 
$\Phi:\OM_\alpha\times(0,3\rho/4)\lrarw\smoo^\alpha(\bdy\dee;\rl)$.
We refer back to the beginning of this proof to compute that the total derivative of $\Phi$ at the
point $(\psi,r)$ has the matrix representation
\[
D\Phi(\psi,r) \ = \ \left[\,\Lambda_{r}+O(r^m) \; \; \; \partial_r\Phi(\psi_{r},r) \ \right]
	: \smoo^\alpha(\bdy\dee;\rl)\oplus\rl\lrarw\smoo^\alpha(\bdy\dee;\rl),
\]
where $\Lambda_{r}$ is as defined in Step~2, and $\partial_r\Phi$ denotes the partial
Fr{\'e}chet derivative with respect to $r$. It is easy to show that the latter exists, and that
$D\Phi$ varies continuously with $(\psi,r)\in\OM_\alpha\times(0,3\rho/4)$.
Shrinking $R_0$ if necessary, it follows from our Claim in Step~2 that
$\left.\partial_{\psi}\Phi\right|_{(\psi_r,r)}$ is an isomorphism for each $r\in (0,R_0)$.
The reader is now referred to the note at the beginning of Step~2. In view of our last assertion,
and the fact that $\Phi(\psi_r,r)=0$, we can now apply the Implicit Function Theorem. For 
a given $r^0\in(0,R_0)$, there exist a $\smoo^\alpha$-open neighbourhood $\omega(r^0)$ and 
an interval $I(r^0)\Subset(0,R_0)$ containing $r^0$ such that:
\begin{itemize}
 \item For each $r\in I(r^0)$, there exists a {\em unique} $\psi\in \omega(r^0)$
 such that $\Phi(\psi,r)=0$.
 \item If we designate this $\psi$ as $\gamma_{r^0}(r)$, then $\gamma_{r^0}$ is of class
 $\smoo^1$ on $I(r^0)$.
\end{itemize}
We now recall that $H(G(r),r)=G(r) \ \forall r\in(0,R_0)$. Applying this fact to \eqref{E:conti},
we get
\[
 \frac{|G(r)-G(r^0)|}{2} \ \leq \left(K_1 + \frac{K_2}{(r^0)^{m-1}r^{m-1}}\right)|r^0-r|,
\]
whence $\lim_{r\to r^0}G(r)=G(r^0)$. Combining this with the conclusions of the Implicit
Function Theorem, there exists an open interval $I^\prime(r^0)\subseteq I(r^0)$ such that
\[
 \left.G\right|_{I^\prime(r^0)} \ = \ \left.\gamma_{r^0}\right|_{I^\prime(r^0)}.
\]
But as $\gamma_{r^0}$ is $\smoo^1$-smooth, and $r^0\in (0,R_0)$ was picked arbitrarily,
we conclude that $G$ is $\smoo^1$-smooth. Owing to the mapping properties of the Poisson kernel, 
it follows from Fact~B that the H{\"o}lder norms, thus the sup-norms of $\mathfrak{g}(r)$, 
shrink to zero as $r\lrarw 0^+$. \hfill $\Box$

\begin{remark}\label{R:Forst}
It might seem to the reader that since (in the notation of the proof above) 
$\{(\kappa r\widetilde{G},(\kappa r)^m):r\in (0,1)\}$ is a family of analytic discs with
boundaries in $\Gamma(\lead)\setminus\{(0,0)\}$, we could use Forstneri{\v{c}}'s results 
in \cite{forstneric:adbmrsCC87} to give a ``quick'' proof of Theorem~\ref{T:posInd}. However,
the relevant theorems in \cite{forstneric:adbmrsCC87}, i.e. Theorems~1 and 3, are 
{\em non-quantitative}. We would have to augment them with estimates (in a similar spirit
to those in the proof above) to learn whether $\Gamma(\lead+\rem)\setminus\{(0,0)\}$ is
``close enough'' to $\Gamma(\lead)\setminus\{(0,0)\}$ for us to deduce the existence of
the desired family $\{\mathfrak{g}(r):r\in (0,R_0)\}$ (especially the existence of
$\mathfrak{g}(r)$'s {\em arbitrarily close} to the CR~singularity).
\end{remark}

\section{A comparison of Theorem~\ref{T:posInd} with previous results}\label{S:nonSubh}

Theorem~\ref{T:posInd} is reminiscent of some of results in \cite{wiegerinck:lpchdcrs95} 
about the existence of analytic discs in the polynomially-convex hull around a
degenerate CR~singularity. We paraphrase Wiegerinck's results to the context that we have
been studying.

\begin{result}[{paraphrasing parts of Theorem~3.3 and 3.5, \cite{wiegerinck:lpchdcrs95}}]\label{R:Wiegerinck}
Let $\varphi$ be $\smoo^{m+1}$-smooth function defined in a neighbourhood of $0\in\cplx$
that vanishes to order $m$ at $0$. Write
\[
\varphi(z) \ = \ \lead(z)+\rem(z) \quad(\text{with $|z|$ sufficiently small}),
\]
where $\lead$ is polynomial that is homogeneous of degree $m$, and $\rem(z)=O(|z|^{m+1})$.
Suppose $(0,0)$ is an isolated CR~singularity of $\Gamma(\varphi)$ and that $\lead$ 
is real-valued. If $\MasInd(\Gamma(\varphi),0)>0$ and $\lead$ is a subharmonic,
non-harmonic function, then $\Gamma(\varphi)$ is not locally polynomially convex at 
$(0,0)$.
\end{result}   
\smallskip

Results like the above rely strongly on the results of Chirka~\&~Shcherbina
\cite{chirkaShch:prd&fhg95} (also refer to \cite{shch:phg93} by Shcherbina), which 
can be used to analyse the structure of the polynomially-convex hulls of graphs of functions
defined on certain classes of sets in $\CC$ that are homeomorphic to the $2$-sphere. 
The potential-theoretic ideas used in \cite{shch:phg93} and \cite{chirkaShch:prd&fhg95}
shape the hypotheses of the results therein. Those hypotheses lead to certain subharmonicity
conditions being imposed in the results of \cite{wiegerinck:lpchdcrs95}. For example, in the
setting of Result~4.1, they translate into the requirement that $\lead$ be a subharmonic, 
non-harmonic function. This raises the following question: {\em with the hypotheses imposed 
on $\lead$ in Theorem~\ref{T:posInd}, is it possible that $\lead$ is automatically subharmonic?} 
If this were the case, then Theorem~\ref{T:posInd} would be a special case of the results in
\cite{wiegerinck:lpchdcrs95}.
\smallskip

We demonstrate in this section that the answer to the above question is negative. There are
pairs $(\srf,p)$, where $p$ is an isolated degenerate CR~singularity, to which Wiegerinck's
hypotheses do not apply but which admit Bishop discs.
The point of Theorem~\ref{T:posInd} was to demonstrate some techniques
for examining the local polynomially-convex hull near an isolated CR~singularity that do not 
require any subharmonicity-type conditions. We now present a one-parameter family of relevant
counterexamples. 

\begin{example}\label{Ex:nonSubh}
For each $C\in(1/3,2/3)$, there exists an $\eps_C>2/3$ such that the real-valued, homogeneous
polynomial
\[
\leadC(z) \ := \ \frac{C}{2}(z^4+\zbar^4)+\eps_C(z^3\zbar+z\zbar^3)+|z|^4
\]
has the following properties:
\begin{itemize}
\item[a)] $0$ is an isolated CR~singularity of $\Gamma(\leadC)$ satisfying
$\MasInd(\Gamma(\leadC),0)>0$; and
\item[b)] $\lead$ is ${\rm not}$ subharmonic.
\end{itemize}
\end{example}

\noindent{To arrive at a polynomial with the above properties, let us first examine
\[
\mathscr{F}(z;\eps,C) \ := \ \frac{C}{2}(z^4+\zbar^4)+\eps(z^3\zbar+z\zbar^3)+|z|^4.
\]
Then
\[
\bdy^2_{z\zbar}\mathscr{F}(z;\eps,C) \ = \ 3\eps(z^2+\zbar^2)+4|z|^2 \ = \ (6\eps\cos2\theta+4)|z|^2,
\]
where, as usual, we write $z=|z|\cis{\theta}$. Then, clearly
\begin{equation}\label{E:subhFail}
\text{\em $\mathscr{F}(\bcdot;\eps,C)$ fails to be subharmonic} \ \iff \ \eps>2/3.
\end{equation}
From Lemma~\ref{L:key}, we realise that
\begin{equation}\label{E:positivity}
\MasInd(\Gamma(\leadC),0)>0 \ \iff \ \leadC(\cis{\theta})\neq 0 \;\; \forall\theta\in\rl.
\end{equation}
Hence, to begin with, we shall examine whether there are any values of the parameter $C$
such that
\[
\mathscr{F}(\cis{\theta};2/3,C) \ = \ 2C(\cos2\theta)^2+(4/3)\cos2\theta+(1-C) \ > \ 0 
\;\; \forall\theta\in\rl.
\]
We will then perturb the parameter $\eps$ away from $\eps=2/3$ so as to ensure that positivity 
is preserved, {\em but subharmonicity fails}. To this end, we set $X:=\cos2\theta$ in the above
inequality
to get
\begin{equation}\label{E:posIndIneq}
2CX^2+(4/3)X+(1-C) \ > 0.
\end{equation}
Note that:
\begin{align}
\eqref{E:posIndIneq} \ \iff& \ \begin{cases}
                C>0, \  \text{and} \\
                (4/3)^2-8C(1-C)<0.
                \end{cases} \notag \\
            \iff& \ C\in(1/3,2/3). \notag
\end{align}
This shows that for each $C\in(1/3,2/3)$, $\mathscr{F}(\bcdot;2/3,C)>0$ on $\cplx\setminus\{0\}$.}
\smallskip

Finally, note that
\begin{itemize}
\item $S^1\times\{2/3\}\times\{C\}$ is a compact subset of $S^1\times(\rl_+)\times(1/3,2/3)$; and
\item $\left.\mathscr{F}\right|_{S^1\times\{2/3\}\times\{C\}}>0$ for each $C\in(1/3,2/3)$.
\end{itemize}
Since $\mathscr{F}$ is continuous on $S^1\times(\rl_+)\times(1/3,2/3)$, there exists a 
$\delta(C)>0$ such that 
\begin{equation}\label{E:perturb}
(\eps,C)\in\mathbb{B}^2((2/3,C);\delta(C)) \ \Longrightarrow \ \mathscr{F}(\cis{\bcdot};\eps,C)>0.
\end{equation}
We now pick an $\eps_C\in(2/3,2/3+\delta(C))$ and define $\leadC:=\mathscr{F}(\bcdot;\eps_C,C)$.
From \eqref{E:positivity}, \eqref{E:posIndIneq} and \eqref{E:perturb}, we conclude 
that $\leadC$ satisfies property~(a). 
We have chosen $\eps_C>2/3$; hence, by \eqref{E:subhFail}, $\leadC$
fails to be subharmonic. \hfill $\blacksquare$
\medskip

\section{The proof of Theorem~\ref{T:negInd}}\label{S:negInd}
\medskip

A non-trivial result that we will require is the following theorem by Forstneri{\v{c}}, 
which we shall paraphrase:

\begin{result}[paraphrasing Theorem~2, \cite{forstneric:adbmrsCC87}]\label{R:forstInd}
Let $M$ be a maximally totally-real $\smoo^4$-smooth submanifold of an open subset of $\CC$,
and let $\mathfrak{g}\in A^\alpha(\dee;\CC)$ be an immersed analytic disc with boundary
in $M$ such that the tangent bundle $TM$ is trivial over an $M$-open neighbourhood of
$\mathfrak{g}(\bdy\dee)$. If $\pathInd{{\mathfrak{g}(\cis{\bcdot})}}\leq 0$, then there
is an open neighbourhood $\OM\subset A^\alpha(\dee;\CC)$ of $\mathfrak{g}$ such that
the only analytic discs $F\in\OM$ with boundary in $M$ are of the form 
$\mathfrak{g}\circ\varphi$, where $\varphi\in{\rm Aut}(\dee)$.
\end{result} 

\begin{remark}\label{Rem:forstInd}
Theorem~2 in \cite{forstneric:adbmrsCC87} has been stated --- in the notation of
Result~\ref{R:forstInd} --- only for $\mathfrak{g}\in A^{1/2}(\dee;\CC)$. However,
the observations made in \cite[Remark~1]{forstneric:adbmrsCC87} about Theorem~1 apply
as well to Theorem~2 in \cite{forstneric:adbmrsCC87}. In other words, we can allow
$\mathfrak{g}\in A^\alpha(\dee;\CC)$ in the hypothesis of the latter theorem.
\end{remark}
\smallskip
 
We are now ready to provide 
\smallskip

\noindent{{\bf The proof of Theorem~\ref{T:negInd}.} We first consider Part~(1). Let 
$(\srf,p)$ be as described in the hypothesis of the theorem. As before, we may work with the
graph $\Gamma(\lead+\rem)$, where $\lead$ and $\rem$ have the same meanings as in \eqref{E:normform2*}.
Since $\MasInd(\srf,p)$ is invariant under a holomorphic change of coordinate, arguing exactly as
in the proof of Theorem~\ref{T:posInd}, $\MasInd(\Gamma(\lead),0)\leq 0$. By hypothesis, and
the formula \eqref{E:keyWindex} in Lemma~\ref{L:key}, we conclude that $\lead$ changes sign. To see 
this, we rely on the fact that $(0,0)$ is an isolated CR~singularity of $\Gamma(\lead)$. We have
discussed that in this case --- see equation \eqref{E:zbar} --- if $\lead(\cis{\bcdot})$ has zeros,
then it has only simple zeros. This fact --- combined with the fact that, by the formula 
\eqref{E:keyWindex}, $\lead(\cis{\bcdot})^{-1}\{0\}\neq\emptyset$ --- implies that $\lead$
must change sign. Then,
each level set $\lead^{-1}\{c\}, \ c\in\rl$, is a finite union of disjoint arcs in $\cplx$.}
\smallskip

Since $\rem(z)=O(|z|^{m+1})$, there exists a $\delta>0$ which is sufficiently small that the level 
sets of $\left.(\lead+\rem)\right|_{\overline{D(0;\delta)}}$, i.e. the sets
\begin{equation}\label{E:levels}
\{z\in \overline{D(0;\delta)}:(\lead+\rem)(z)=c\}
\end{equation}
do not separate $\cplx$ for each $c>0$. We now appeal to the
following:

\begin{result}[Theorem~1.2.16, \cite{stout:pc07}]\label{R:Mergelyan}
If $X\subset\Cn$ is compact and if $\mathscr{P}(X)$ contains a real-valued function $f$, then
$X$ is polynomially convex if and only if each fibre $f^{-1}\{c\}$, $c\in\rl$, is polynomially convex.
\end{result}

\noindent{We clarify that, for $X\subset\Cn$ compact,
\begin{align}
\mathscr{P}(X) \ &:= \ \text{the uniform algebra on $X$ generated by} \notag\\  
                &\qquad\qquad\text{the class
                $\{P|_X:P\in\cplx[z_1,\dots,z_n] \ \}$}.\notag
\end{align}
Taking $X:= \Gamma(\lead+\rem; \overline{D(0;\delta)})$ and $f(z,w):=w$, and observing that
each of the sets in \eqref{E:levels} is polynomially convex, we conclude from Result~\ref{R:Mergelyan}
that $\Gamma(\lead+\rem)$ is locally polynomially convex at $(0,0)$ --- or, equivalently, that
$\srf$ is locally polynomially convex at $p$.}
\smallskip

We now consider Part~(2). As before, we shall work in the coordinate system $(z,w)$ with
respect to which $(\srf,p)$ has the representation \eqref{E:normform2*}. Suppose, for some
$\alpha\in(0,1)$, there exists a continuous
one-parameter family $\mathfrak{g}:(0,1)\lrarw A^\alpha(\dee;\CC)$ of immersed, non-constant
analytic discs with the following three properties:
\begin{itemize}
\item[a)] $\mathfrak{g}(t)(\bdy\dee)\subset(\srf\setminus\{p\})\cap U_p \ \forall t\in(0,1)$.
\item[b)] $\mathfrak{g}(t)(\cis{\bcdot})$ is a simple closed curve in $\srf \
\forall t\in(0,1)$.
\item[c)] $\mathfrak{g}(t)(\zt)\lrarw\{p\}$ for each $\zt\in\overline{\dee}$ as $t\lrarw 0^+$.
\end{itemize}
Let $\delta_0>0$ be so small that for every smooth, positively-oriented, simple closed path 
$\gamma: S^1\lrarw D(0;\delta_0)\setminus\{(0,0)\}$,
\begin{equation}\label{E:approxWind}
{\sf Wind}\left(\partl{(\lead+\rem)}{\zbar}\circ\gamma,0\right) \ = \
{\sf Wind}\left(\partl{\lead}{\zbar}\circ\gamma,0\right).
\end{equation}
Such a $\delta_0>0$ exists because $\rem(z)=O(|z|^{m+1})$. Then, in view of Lemma~\ref{L:windex}
and the properties (a)--(c), \eqref{E:approxWind} allows us to infer that there exists a 
$t_0\in(0,1)$ such that 
\begin{multline}
\qquad \pathInd{{\mathfrak{g}(t)(\cis{\bcdot})}}(\Gamma(\lead+\rem))
 \ = \ \MasInd(\Gamma(\lead+\rem),0) \\ 
= \ \MasInd(\Gamma(\lead),0) \ \leq \ 0, \;\; \forall t\in(0,t_0). \qquad {} \notag
\end{multline}
We pick a $t^*\in(0,t_0)$. Since $\srf$ is now assumed to be $\smoo^4$-smooth, we can 
apply Result~\ref{R:forstInd}. By this result, $\exists\eps>0$ such that
for any analytic disc $F\in A^\alpha(\dee;\CC)$ with boundary in 
$\Gamma(\lead+\rem)\setminus\{(0,0)\}$ such that $0<\hnrm{F-\mathfrak{g}(t^*)}<\eps$,
$F=\mathfrak{g}(t^*)\circ\varphi$, where $\varphi\in{\rm Aut}(\dee)$. However, this
leads to a contradiction because, owing to the continuity of $\mathfrak{g}$ and to
(c) above, there must exist a $t^\prime\in(0,t_0)$, $t^\prime\neq t^*$, such that
\[
0 \ < \ \hnrm{\mathfrak{g}(t^*)-\mathfrak{g}(t^\prime)} \ < \ \eps, \; \; \text{and} \;
\; {\rm Image}(\mathfrak{g}(t^*)) \ \neq \ {\rm Image}(\mathfrak{g}(t^\prime)).
\]
Hence, our assumption about the existence of $\mathfrak{g}:(0,1)\lrarw A^\alpha(\dee;\CC)$
must be wrong, which establishes Part~(2). \hfill $\Box$ 
\bigskip

\noindent{{\bf Acknowledgements.} A large part of the preliminary work on this 
article --- looking up the extensive literature on the subject, in particular --- was
done while on a visit to the ICTP, Trieste, during February--March, 2008. The support and
hospitality of the ICTP, and the facilities provided by the ICTP library, are warmly acknowledged.
I am grateful to Nessim Sibony for suggesting a simple proof of Part~(1) of Theorem~\ref{T:negInd}
during a visit to the Universit{\'e} Paris-Sud 11 in 2009. This visit was supported by the ARCUS 
programme (French Foreign Ministry/R{\'e}gion Ile-de-France).}

\end{document}